\documentclass[12pt,oneside]{amsart}
\usepackage{geometry}                
\usepackage{graphicx}
\usepackage{amssymb}
\usepackage{epstopdf}

\usepackage{amsmath,amsthm,amscd,amssymb}
\usepackage{latexsym}
\usepackage[colorlinks,citecolor=red,pagebackref,hypertexnames=false]{hyperref}
\numberwithin{equation}{section}
\theoremstyle{plain}
\newtheorem{theorem}{Theorem}[section]
\newtheorem{lemma}[theorem]{Lemma}

\theoremstyle{definition}
\newtheorem{definition}[theorem]{Definition}

\theoremstyle{remark}

\newtheorem{case[theorem]}{Case}

\def \F {\mathbb{F}}

\date{November 6, 2017}      
\author{P. Birklbauer, A. Iosevich and T. Pham}
\address{Department of Mathematics, University of Rochester, Rochester, NY}
\email{philipp.birklbauer@rochester.edu}
\address{Department of Mathematics, University of Rochester, Rochester, NY}
\email{iosevich@math.rochester.edu}
\address{Department of Mathematics, University of California, San Diego}
\email{v9pham@ucsd.edu}

\thanks{The second listed author was partially supported by the NSA Grant H98230-15-1-0319. The third listed author was supported by Swiss National Science Foundation grant P2ELP2-175050.}
\title{Distances from points to planes} 
\begin{document}
\maketitle
\begin{abstract} We prove that if $E \subset {\Bbb F}_q^d$, $d \ge 2$, $F \subset \operatorname{Graff}(d-1,d)$, the set of affine $d-1$-dimensional planes in ${\Bbb F}_q^d$, then $|\Delta(E,F)| \ge \frac{q}{2}$ if $|E||F|>q^{d+1}$, where $\Delta(E,F)$ the set of distances from points in $E$ to lines in $F$. In dimension three and higher this significantly improves the exponent obtained by Pham, Phuong, Sang, Vinh and Valculescu (\cite{PPSVV17})  \end{abstract}  

\maketitle

\section{Introduction} 

The Erd\H os-Falconer distance problem in ${\Bbb F}_q^d$ is to determine how large $E \subset {\Bbb F}_q^d$ needs to be to ensure that 
$$\Delta(E)=\{||x-y||: x,y \in E\},$$ with $||x||=x_1^2+x_2^2+\dots+x_d^2$, is the whole field ${\Bbb F}_q$, or at least a positive proportion thereof. Here and throughout, ${\Bbb F}_q$ denotes the field with $q$ elements and ${\Bbb F}_q^d$ is the $d$-dimensional vector space over this field. 

The distance problem in vector spaces over finite fields was introduced by Bourgain, Katz and Tao in \cite{BKT04}. In the form described above, it was introduced by the second listed author of this paper and Misha Rudnev (\cite{IR07}), who proved that $\Delta(E)={\Bbb F}_q$ if $|E|>2q^{\frac{d+1}{2}}$. It was shown in \cite{HIKR11} that this exponent is essentially sharp for general fields when $d$ is odd. When $d=2$, it was proved in \cite{CEHIK10} that if if $E \subset {\Bbb F}_q^2$ with $|E| \ge cq^{\frac{4}{3}}$, then $|\Delta(E)| \ge C(c)q$. We do not know if improvements of the $\frac{d+1}{2}$ exponent are possible in even dimensions $\ge 4$. We also do not know if improvements of the $\frac{d+1}{2}$ exponent are possible in any even dimension if we wish to conclude that $\Delta(E)={\Bbb F}_q$, not just a positive proportion. 

More generally, let $\operatorname{Graff}(k,d)$ denote the set of $k$-dimensional affine planes in ${\Bbb F}_q^d$. In this paper we shall focus on distances from points in subsets of $\operatorname{\operatorname{Graff}}(0,d)={\Bbb F}_q^d$ to $d-1$-dimensional planes in subsets of $\operatorname{\operatorname{Graff}}(d-1,d)$. The set of distances from points to points (see e.g. \cite{IR07}) can be defined as the set of equivalence classes of two-point configurations where two pairs $(x,y)$ and $(x',y')$ are equivalent if there exists a translation $\tau \in {\Bbb F}_q^d$ and a rotation $\theta \in O_d({\Bbb F}_q^d)$ that takes one pair to the other. In the case of points and $d-1$-dimensional planes in ${\Bbb F}_q^d$, we may similarly define $(x,h)$ and $(x',h')$ to be equivalent, where $x$'s are points and $h$'s are planes, if after translation $x$ to $x'$, there exists a rotation $\theta \in O_d({\Bbb F}_q)$ that takes $h$ to $h'$. Denote the resulting set of equivalence classes by $\Delta(E,F)$. 

Before stating our main results, we need to say a few words about the parameterization of $(d-1)$-dimensional planes in ${\Bbb F}_q^d$. A $(d-1)$-dimensional plane in ${\Bbb F}_q^d$ can be expressed in the form 
$$ H_{v,t}=\{y \in {\Bbb F}_q^d: y \cdot v=t \}, $$ where we should think of $v$ as a normal vector to the plane and $t$ as the distance to the origin. Note that the notion of distance from a point to a plane described above only makes sense if $||v|| \not=0$. We shall henceforth refer to planes with this property as {\it non-degenerate} planes. See Lemma \ref{invariance} below. 

\begin{definition} We say that $V \subset {\Bbb F}_q^d$ is a direction set if given $x \in {\Bbb F}_q^d$, $x \not=\vec{0}$, there exists $v \in V$ and 
$t \in {\Bbb F}_q^{*}$ such that $x=tv$. \end{definition} 

It is very convenient to work with a "canonical" direction set provided by the following simple observation. 

\begin{lemma} \label{normalstructure} Let $S_t=\{x \in {\Bbb F}_q^d: ||x||=t \}$. Let $\gamma \in {\Bbb F}_q^{*}$ be a non-square. Define $V_{\gamma}=S_0 \cup S_1 \cup S_{\gamma}$. Then $V_{\gamma}$ is a direction set. \end{lemma} 

\vskip.125in 

To prove this, choose $x$ such that $||x||=0$. Then $x \in S_0$. Now choose $x$ such that $||x||=t^2$ for some $t \not=0$. Then ${\left( \frac{x_1}{t} \right)}^2+{\left( \frac{x_2}{t} \right)}^2=1$, so $x=tv$ with $v \in S_1$. Finally, suppose that $||x||=u$ where $u$ is not a square in ${\Bbb F}_q^{*}$. To see that $x=tv$ for some $v \in S_{\gamma}$, it is enough to check that $u\gamma^{-1}$ is a square in ${\Bbb F}_q$. Moreover, it is enough to prove that a product of two non-squares is a square. To see this, let $\phi: {\Bbb F}_q^{*} \to {\Bbb F}_q^{*}$ given by $\phi(x)=ux$, where $u$ is a non-square. The image of a square is certainly a non-square since otherwise $u$ would be forced to be a square. It follows that an image of a non-square is a square since exactly half the elements of ${\Bbb F}_q^{*}$ are squares. This completes the proof of Lemma \ref{normalstructure}. 

\vskip.125in 

Our main result is the following. 

%

\begin{theorem} \label{maind}  Let $E \subset {\Bbb F}_q^d$, $d \ge 2$, and $F$ be a subset of non-degenerate planes in $\operatorname{Graff}(d-1,d)$, $d \ge 2$. Let $\gamma$ be a non-square in ${\Bbb F}_q$. Suppose that $|E||F|>q^{d+1}$. Then $|\Delta(E,F)|\gg q$. More precisely, 
$$ |\Delta(E,F)| \ge \frac{ {|E|}^2{|F|}^2}{2{|E|}^2{|F|}^2q^{-1}+2q^{d-1}|E||F| \cdot \max_{||v||=1,\gamma} \sum_t F(v,t)}.$$ 
\end{theorem} 

\vskip.125in 

When $d=2$, a better exponent was obtained by Pham, Phuong, Sang, Vinh and Valculescu (\cite{PPSVV17}). They proved that the conclusion of Theorem \ref{maind} holds in ${\Bbb F}_q^2$ if $|E||F|>Cq^{\frac{8}{3}}$. 

It is not clear if it is possible to weaken the $|E||F|>q^{d+1}$ assumption in higher dimensions. It is not difficult to see that we cannot do better than assuming $|E||F|>q^d$. To see this, take $q=p^2$, $p$ prime, let $E={\Bbb F}_p^d$ and $F$ be the set of all $(d-1)$-dimensional affine planes in ${\Bbb F}_p^d$. Then $|E| \approx |F| \approx q^{\frac{d}{2}}$ while $\Delta(E,F)=p$. 

\vskip.25in 

\section{Proof of Theorem \ref{maind}} 
\vskip.125in 
We begin with a couple simple algebraic observations that make working with $\Delta(E,F)$ much easier. Given $F \subset \operatorname{Graff}(d-1,d)$, we write the indicator function of $F$ in the form $F(v,t)$, where each plane in $\operatorname{\operatorname{Graff}}(d-1,d)$ is parameterized by $(v,t) \in V_{\gamma} \times {\Bbb F}_q$, where $V_{\gamma}$ is as in Lemma \ref{normalstructure}. For a point $x\in E$ and a plane $F(v, t)\in F$, the distance function between them, denoted by $d[x, F(v, t)]$, is defined by
\[d[x, F(v, t)]:=\frac{(x \cdot v-t)^2}{||v||}.\]
In the following lemma, we show that the size of $\Delta(E, F)$ is at least the number of distinct non-zero distances between points in $E$ and planes in $F$. 
\begin{lemma} \label{invariance} Let $F \subset \operatorname{\operatorname{Graff}}(d-1,d)$ be parameterized as above, with coordinates $(v,t) \in V_{\gamma} \times {\Bbb F}_q$, where $||v|| \not=0$. Then 
\[ |\Delta(E,F)|\ge \#\left\{\frac{(x \cdot v-t)^2}{||v||}\ne 0: x \in E; (v,t) \in F\right \}.\] 
\end{lemma} 
\begin{proof}
To prove this lemma, it is enough to indicate that for $x,x' \in E$ and $(v,t), (v',t') \in F$, if $d[x, F(v, t)]=d[x', F(v', t')]$, then there is a rotation $\theta$ such that the translation from $x$ to $x'$ followed by $\theta$ takes the plane $F(v, t)$ to $F(v', t')$. Indeed, since $d[x, F(v, t)]=d[x', F(v', t')]$, we have
\begin{equation}
\label{invariance::equal}
\frac{(x \cdot v-t)^2}{||v||}= \frac{(x' \cdot v'-t')^2}{||v'||}.
\end{equation}
This implies that $||v||/||v'||$ is a square. From this we deduce, just as in the proof of Lemma \ref{normalstructure} above that either both $||v||$ and $||v'||$ are squares or they are both non-squares. Since we are only considering $||v||$ and $||v'||$ that are equal to $1$ or $\gamma$, we conclude that $||v||=||v'||$. From the equation (\ref{invariance::equal}), we have $x \cdot v-t= \pm (x' \cdot v'-t')$. Without loss of generality, we assume that $x'=0$. Since $||v||=||v'||\ne 0$, there exists a rotation $\theta \in O_d(\F_q)$ such that $\theta v=\pm v'$. Thus we have the following
 \begin{align*}
\{ \theta(y-x) : y\cdot v = t\} &= \{ z : (\theta^{-1}z+x) \cdot v = t\} \\ & = \{z : z \cdot \theta v = t- x \cdot v \} \\ & = \{z : \pm z \cdot v' =  \pm t' \} \\ &  = \{z : z \cdot v' = t' \}.
\end{align*}
In other words, the translation from $x$ to $x'$ followed by the rotation $\theta$ about $x'$ takes the plane $F(v, t)$ to the plane $F(v', t')$. This concludes the proof of the lemma.
\end{proof}

\vskip.125in 

Before proving Theorem \ref{maind}, we need to review the fourier transform of functions on $\mathbb{F}_q^d$. Let $\chi$ be a non-trivial additive character on $\mathbb{F}_q$. For a function $f: \mathbb{F}_q \to \mathbb{C}$, we define 
\[\widehat{f}(m)=q^{-d} \sum_{x \in \mathbb{F}_q^d} \chi(-x \cdot m) f(x).\]
It is clear that\[ f(x)=\sum_{m \in \mathbb{F}_q^d} \chi(x \cdot m) \widehat{f}(m),\] and
\[ \sum_{m \in \mathbb{F}_q^d} {|\widehat{f}(m)|}^2=q^{-d} \sum_{x \in \mathbb{F}_q^d} {|f(x)|}^2.\]

We are now ready to prove Theorem \ref{maind}.
\begin{proof}[Proof of Theorem \ref{maind}]
It follows from Lemma \ref{invariance} that it suffices to prove that 
\[\#\left\{\frac{(x \cdot v-t)^2}{||v||}: x \in E; (v,t) \in F\right \} \ge \frac{ {|E|}^2{|F|}^2}{2q^{-1}{|F|}^2{|E|}^2+2q^{d-1} \max_{v \in V_{\gamma}} F(v,t) \cdot |E||F|}.\]
For $r\in \mathbb{F}_q$, let 
$$ \nu(r):=\sum_{(x \cdot v-t)^2=r||v||} E(x)F(v,t).$$ 
By the Cauchy-Schwartz inequality, 
$$ {|E|}^2{|F|}^2={\left(\sum_r \nu(r) \right)}^2 \leq \sum_{r\in \mathbb{F}_q} \nu^2(r)\cdot \#\left\{\frac{(x \cdot v-t)^2}{||v||}: x \in E; (v,t) \in F\right \}.$$ 
This implies that 
\[ \#\left\{\frac{(x \cdot v-t)^2}{||v||}: x \in E; (v,t) \in F\right \}\ge \frac{|E|^2|F|^2}{\sum_{r\in \mathbb{F}_q}\nu(r)^2}.\]
We now are going to show that
\[\sum_{r\in \mathbb{F}_q}\nu(r)^2\le 2q^{-1}{|F|}^2{|E|}^2+2q^{d-1}|F||E| \cdot \max_{v \in V} \sum_t F(v,t).\]
Indeed, applying Cauchy-Schwarz inequality again gives us
\begin{align*}
 \sum_{r\in \mathbb{F}_q} \nu^2(r) &\leq |F| \sum_{\substack{x, x', v, t\\ d[x, F(v, t)]=d[x', F(v, t]}} E(x)E(x')F(v, t)\\
&= |F|\left(\sum_{x \cdot v-x' \cdot v=0} F(v,t) E(x)E(x')+\sum_{x \cdot v+x' \cdot v-2t=0} F(v,t) E(x)E(x')\right)=|F|(I+II).
\end{align*}
We now bound $I$ and $II$ as follows.
\begin{align}\label{keymoment1}
I=\sum_{x \cdot v-x' \cdot v=0} F(v,t) E(x)E(x')&=q^{-1}{|F|}{|E|}^2+q^{-1}\sum_{s \not=0} \sum_{v, t, x, x'} \chi(sv \cdot (x-x')) F(v,t) E(x)E(x')\nonumber\\
&=q^{-1}{|F|}{|E|}^2+q^{-1}\sum_{s \not=0} \sum_{v, t, x, x'} \chi(sv \cdot (x-x')) F(v,t) E(x)E(x')\nonumber\\
&=q^{-1}{|F|}{|E|}^2+q^3 \sum_{s \not=0} \sum_{v,t} {|\widehat{E}(sv)|}^2 F(v,t)\nonumber\\
&\leq q^{-1}{|F|}{|E|}^2+q^{2d-1} \cdot \max_{v \in V} \sum_t F(v,t) \cdot \sum_{z \in \mathbb{F}_q^d} {|\widehat{E}(z)|}^2\nonumber\\
&=q^{-1}{|F|}{|E|}^2+q^{d-1}|E| \cdot \max_{v \in V} \sum_t F(v,t),
\end{align}
where we used $\sum_{z\in \mathbb{F}_q^d}{|\widehat{E}(z)|}^2=q^{-d}|E|.$
\begin{align}\label{keymoment2}
II=\sum_{x \cdot v-x' \cdot v=2t} F(v,t) E(x)E(x')&=q^{-1}{|F|}{|E|}^2+q^{-1}\sum_{s \not=0} \sum_{v, t, x, x'} \chi(sv \cdot (x+x'))\chi(2st) F(v,t) E(x)E(x')\nonumber\\
&=q^{-1}{|F|}{|E|}^2+q^{-1}\sum_{s \not=0} \sum_{v, t, x, x'} \chi(sv \cdot (x+x')) \chi(2st)F(v,t) E(x)E(x')\nonumber\\
&=q^{-1}{|F|}{|E|}^2+q^3 \sum_{s \not=0} \sum_{v,t} {\widehat{E}(sv)}{\widehat{E}(sv)}\chi(st+st) F(v,t)\nonumber\\
&\le q^{-1}{|F|}{|E|}^2+q^{2d-1} \sum_{s \not=0} \sum_{v,t} {|\widehat{E}(sv)|}^2 F(v,t)\nonumber\\
&\leq q^{-1}{|F|}{|E|}^2+q^{2d-1}\cdot \max_{v \in V} \sum_t F(v,t) \cdot \sum_{z \in \mathbb{F}_q^d} {|\widehat{E}(z)|}^2\nonumber\\
&=q^{-1}{|F|}{|E|}^2+q^{d-1}|E| \cdot \max_{v \in V} \sum_t F(v,t).
\end{align}
\vskip.125in 
Putting (\ref{keymoment1}) and (\ref{keymoment2}) together, we obtain
\[\sum_{r\in \mathbb{F}_q}\nu(r)^2\le 2q^{-1}{|F|}^2{|E|}^2+2q^{d-1}|F||E| \cdot \max_{v \in V} \sum_t F(v,t).\]
\vskip.125in

We conclude that 
\[\#\left\{\frac{(x \cdot v-t)^2}{||v||}: x \in E; (v,t) \in F\right \} \ge \frac{ {|E|}^2{|F|}^2}{2q^{-1}{|F|}^2{|E|}^2+2q^{d-1} \max_{v \in V_{\gamma}} F(v,t) \cdot |E||F|}.\]
Hence, 
$$ |\Delta(E,F)| \ge  \frac{ {|E|}^2{|F|}^2}{2q^{-1}{|F|}^2{|E|}^2+2q^{d-1} \max_{v \in V_{\gamma}} F(v,t) \cdot |E||F|}.$$ 
This concludes the proof once we note that 
$$ \max_{v \in V_{\gamma}} F(v,t) \leq q.$$

\end{proof}


\vskip.25in 

\begin{thebibliography}{40}


\bibitem{BKT04} J. Bourgain, N. Katz, T. Tao. {\it A sum-product estimate in finite fields, and applications.} Geom. Funct. Anal. \textbf{14} (2004), 27--57.

\bibitem{CEHIK10} J. Chapman, M. B. Erdogan, D. Hart, A. Iosevich, and D. Koh, {\it Pinned distance sets, k-simplices, Wolff's exponent in finite fields and sum-product estimates}, Math Z. \textbf{271} (2012), 63--93. 

%




















\bibitem{HIKR11}
D. Hart, A. Iosevich, D. Koh, and M. Rudnev {\it Averages over hyperplanes, sum-product theory in vector spaces over finite fields and the Erd\H os-Falconer distance conjecture}, Trans. Amer. Math. Soc. 363 (2011), no. 6, 3255--3275. 




\bibitem{IR07} A. Iosevich and M. Rudnev {\it Erd\H os distance problem in vector spaces over finite fields}, Transactions of the AMS, (2007). 









\bibitem{PPSVV17} T. Pham, N. D. Phuong, N. M. Sang, C. Valculescu, L. A. Vinh, {\it Distinct distances between points and lines in $F_q^2$}, Forum Math., (accepted for publication) (2017).












\end{thebibliography}
\end{document}